\documentclass[11pt]{amsart}
\usepackage[sc]{mathpazo}

\usepackage[margin=1in]{geometry}
\usepackage{amssymb}
\usepackage{hyperref}
\usepackage{mathrsfs}
\usepackage{amsmath,amscd}
\usepackage{color}

\newtheorem{theorem}{Theorem}[section]
\newtheorem{lemma}[theorem]{Lemma}
\newtheorem{proposition}[theorem]{Proposition}
\newtheorem{corollary}[theorem]{Corollary}

\theoremstyle{definition}

\theoremstyle{remark}

\numberwithin{equation}{section}

\def\ZZ{\mathbb{Z}} \def\RR{\mathbb{R}}
 
\def\dgamma{\overline{\gamma}}
\def\ds{\rule{0pt}{1.5ex}}

\begin{document}

\title{Domination ratio of a family of integer distance digraphs with arbitrary degree}
\author{Jia Huang}
\address{Department of Mathematics and Statistics, University of Nebraska at Kearney, Kearney, Nebraska, USA}
\curraddr{}
\email{huangj2@unk.edu}
\thanks{
}

\begin{abstract} 
An integer distance digraph is the Cayley graph $\Gamma(\mathbb{Z},S)$ of the additive group $\mathbb{Z}$ of all integers with respect to a finite subset $S\subseteq\ZZ$. The domination ratio of $\Gamma(\mathbb{Z},S)$, defined as the minimum density of its dominating sets, is related to some number theory problems, such as tiling the integers and finding the maximum density of a set of integers with missing differences. We precisely determine the domination ratio of the integer distance graph $\Gamma(\ZZ,\{1,2,\ldots,d-2,s\})$ for any integers $d$ and $s$ satisfying $d\ge2$ and $s\notin[0,d-2]$. Our result generalizes a previous result on the domination ratio of the graph $\Gamma(\mathbb{Z},\{1,s\})$ with $s\in\mathbb{Z}\setminus\{0,1\}$ and also implies the domination number of certain circulant graphs $\Gamma(\mathbb{Z}_n,S)$, where $\mathbb{Z}_n$ is the finite cyclic group of integers modulo $n$ and $S$ is a subset of $\mathbb{Z}_n$.
\end{abstract}

\maketitle

\section{Introduction}\label{sec:intro}

Let $\Gamma=(V,E)$ be a \emph{digraph}, where $V$ is a set and $E$ is a subset of $V\times V$.
The elements of $V$ are called \emph{vertices}, and an element $(u,v)\in E$ is called a \emph{(directed) edge} from a vertex $u$ to a vertex $v$ as it can be regarded as an arrow from $u$ to $v$.
The digraph $\Gamma$ is said to be \emph{finite} if $V$ and $E$ are both finite, or \emph{infinite} otherwise.
If $(u,v)\in E\Longleftrightarrow (v,u)\in E$ holds for all $u,v\in V$ then we may view $\Gamma$ as an undirected graph by replacing each pair of opposite edges $(u,v)$ and $(v,u)$ with an undirected edge between $u$ and $v$.

The \emph{outdegree} (\emph{indegree}, resp.) of a vertex $v\in V$ is the number of outgoing edges from $v$ (incoming edges to $v$, resp.). 
Given an integer $d\ge0$, we say that $\Gamma$ is a \emph{regular digraph of degree $d$} if every vertex has outdegree and indegree $d$.
An important regular digraph is the \emph{Cayley graph} $\Gamma(G,S)$ determined by a group $G$ and a subset $S\subseteq G$.
The vertex set of $\Gamma(G,S)$ is $G$ and the edge set of $\Gamma(G,S)$ consists of ordered pairs $(g,gs)$ for all $g\in G$ and for all $s\in S$.
A finite Cayley graph $\Gamma(G,S)$ is a regular digraph of degree $d=|S|$ with exactly $n=|G|$ vertices.

For example, a \emph{circulant (di)graph} is a finite Cayley graph $\Gamma(\ZZ_n,S)$ where $\ZZ_n$ is the finite cyclic group of integers modulo $n$ and $S\subseteq\ZZ_n$.
Due to various nice properties (e.g., symmetry, fault-tolerance, and routing capabilities), circulant graphs have become important topological structures for interconnection networks and have been widely used in telecommunication networks, VLSI design, and distributed computation.

Replacing the finite group $\ZZ_n$ with the infinite group $\ZZ$ of all integers under addition gives an \emph{integer distance (di)graph} $\Gamma(\ZZ,S)$, which is a natural generalization of a circulant graph.
We assume $S$ is finite subset of $\ZZ\setminus\{0\}$ throughout this paper ($0$ is excluded from $S$ to avoid loops).
The independence ratio of an undirected integer distance graph 
is closely related to the chromatic number of this graph and has been extensively studied (see, for example, ~Carraher--Galvin--Hartke--Radcliffe--Stolee~\cite{IndepRatio} and Liu~\cite{Liu}). 

In this paper we investigate domination in the integer distance graph $\Gamma(\ZZ,S)$ with $S$ being a finite subset of $\ZZ\setminus\{0\}$.
In a digraph $\Gamma=(V,E)$, a vertex $u$ \emph{dominates} a vertex $v$ if either $u=v$ or $(u,v)\in E$.
A set $D\subseteq V$ is called a \emph{dominating set} of the digraph $\Gamma$ if every vertex $v\in V$ is dominated by some vertex $u\in D$.
With many practical applications (e.g., resource allocation), the concept of domination has been widely studied.
It is a well-known NP-complete problem in graph theory to determine the smallest cardinality of a dominating set of a finite digraph $\Gamma$, known as the \emph{domination number} $\gamma(\Gamma)$ of  $\Gamma$.
There are other meaningful variations of domination, such as total domination.
See the monograph by Haynes, Hedetniemi, and Slater~\cite{FundDom} for a survey of this field.

If $D$ is a dominating set of a digraph $\Gamma$ such that every vertex of $\Gamma$ is dominated by exactly one vertex in $D$, then $D$ is called an \emph{efficient dominating set} (also known as a \emph{perfect code}). 
This is the most efficient way of domination that one can possibly have.
The existence of an efficient dominating set in a finite Cayley graph has become a popular topic in recent years (see, e.g., Chelvam--Mutharasu~\cite{SubgroupEffDom} and Dejter--Serra~\cite{EffDomCayley}).
For a finite Cayley graph $\Gamma=\Gamma(G,S)$ with $|G|=n$ and $|S|=d$, it is straightforward to check that $\gamma(\Gamma)\ge n/(1+d)$, where the equality holds if and only if $\Gamma$ has an efficient dominating set.
The existence of an efficient dominating set in circulant graphs was studied by Huang--Xu~\cite{BonEffDomVertexTrans}, Kumar--MacGillivray~\cite{EffDomCirc}, Obradovi\'c--Peters--Ru\v{z}i\'c~\cite{EffDomCircChord}, Rad~\cite{DomCirc}, and others. 

The integer distance digraph $\Gamma(\ZZ,S)$ has an efficient dominating set if and only if we can tile the integers by translates of $S\cup\{0\}$ without any overlaps.
The problem of tiling integers is a well-studied problem in number theory (see, for example, Newman~\cite{Newman}  and Coven--Meyerowitz~\cite{TilingIntegers}).
It turns out that an efficient dominating set for $\Gamma(\ZZ,S)$ also achieves the maximum density of a set of integers such that the difference between any two of its elements is not contained $S$.
When there is no efficient dominating set for $\Gamma(\ZZ,S)$, one can try to either construct a maximum set of integers whose elements have differences not in $S$, or cover $\ZZ$ with translates $S\cup\{0\}$ with minimum overlaps.
The former is a problem (with only nonnegative integers considered) initially asked by Motzkin (unpublished) and subsequentially studied by Cantor--Gordon~\cite{CantorGordon}, Haralambis~\cite{Haralambis} and others.
The latter amounts to determining the \emph{domination ratio} $\dgamma(\ZZ,S)$ of $\Gamma(\ZZ,S)$, which is defined as the infimum of the \emph{(lower) density}
\begin{equation}\label{eq:density}
\delta(D) := \liminf_{n\to\infty}  \frac{|D\cap [-n,n]|}{2n+1}. 
\end{equation}
over all dominating sets $D$ of $\Gamma(\ZZ,S)$.

In recent work~\cite{DomRatio} we obtained the following basic properties on the domination ratio of $\Gamma(\ZZ,S)$, where $S$ is a finite subset of $\ZZ\setminus\{0\}$.

\begin{proposition}\cite{DomRatio}\label{prop:basic}
{\normalfont(i)} If $S\subseteq S'\subseteq \ZZ\setminus\{0\}$ then $\dgamma(\ZZ,S') \le \dgamma(\ZZ,S) = \dgamma(\ZZ,-S)$.

{\normalfont(ii)} If $|S|\le1$ then $\dgamma(\ZZ,S)=1/(1+|S|)$. If $2\le |S|<\infty$ then $1/(|S|+1)\le \dgamma(\ZZ,S)\le 1/2$.

{\normalfont(iii)} If $S$ is finite and there exists an efficient dominating set of $\Gamma(\ZZ,S)$ then $\dgamma(\ZZ,S)=1/(|S|+1)$.

{\normalfont(iv)} If $S = \{ i_1(s+1)+1,\ldots,i_s(s+1)+s\}$ with $i_1,\ldots,i_s\in\ZZ$ then $\dgamma(\ZZ,S)=1/(s+1)$.

{\normalfont(v)} If $d$ divides all elements of $S$ then $\dgamma(\ZZ,S/d) = \dgamma(\ZZ,S)$, where $S/d:=\{s/d:d\in S\}$.
\end{proposition}

In this paper we show that the converse to Proposition~\ref{prop:basic} (iii) still holds: 

\begin{theorem}\label{thm:RatioEff}
Let $S\subseteq\ZZ\setminus\{0\}$ with $|S|=d<\infty$.
Then the integer distance digraph $\Gamma(\ZZ,S)$ has an efficient dominating set if and only if its domination ratio is $1/(d+1)$.
\end{theorem}

The main result of our previous work~\cite{DomRatio} is the next theorem.

\begin{theorem}\cite{DomRatio}\label{thm:1s}
{\normalfont(i)} If $k\in\ZZ$ then $\dgamma(\ZZ,\{1,3k+2\}) = 1/3$.

{\normalfont(ii)} If $k$ is a positive integer then $\dgamma(\ZZ,\{1,3k+1\}) = \dgamma(\ZZ,\{1,-3k\}) = (k+1)/(3k+2)$.

{\normalfont(iii)} If $k$ is a positive integer then $\dgamma(\ZZ,\{1,3k\}) = \dgamma(\ZZ,\{1,-3k+1\}) = 2k/(6k-1)$.
\end{theorem}


Theorem~\ref{thm:1s} implies that $\gamma(\ZZ_{3k+2},\{\pm1\}) = \gamma(\ZZ_{3k+2},\{1,2\})=k+1$ and $\gamma(\ZZ_{6k-1},\{1,3k\}) =2k$ for any positive integer $k$, where $\gamma(\ZZ_n,S)$ denotes the domination number of a circulant digraph $\Gamma(\ZZ_n,S)$~\cite[Cor. 3.5]{DomRatio}.
In addition, it suggests that if $n\in\ZZ$ is large and if $s$ is close to $n/2$ then $\gamma(\ZZ_n,\{1,s\})$ is close to $\lceil n/3 \rceil$, which is a known upper bound for $\gamma(\ZZ_n,S)$ with $|S|=2$.

Theorem~\ref{thm:1s} together with Proposition~\ref{prop:basic} (v) precisely determines the domination ratio of $\Gamma(\ZZ,\{s,t\})$ and characterizes when $\Gamma(\ZZ,\{s,t\})$ has an efficient dominating set, where $s$ and $t$ are distinct nonzero integers satisfying $s\mid t$.
The domination ratio of $\Gamma(\ZZ,S)$ is not determined yet when $S=\{s,t\}$ with $s\nmid t$ or when $|S|\ge3$.
Generally speaking, the bigger $S$ is, the more difficult it is the determine the domination ratio of $\Gamma(\ZZ,S)$.

Now we generalize Theorem~\ref{thm:1s} to the integer distance graph $\Gamma(\ZZ,\{1,2,\ldots,d-2,s\}$ for any integers $d\ge2$ and $s\notin[0,d-2]$.
This is a regular digraph of an arbitrary degree $d-1\ge1$.
If $s\equiv-1\pmod d$ then $\dgamma(\ZZ,\{1,2,\ldots,d-2,s\})=1/d$ by Proposition~\ref{prop:basic} (iv).
If $s\not\equiv-1\pmod d$ then we have $s=dk+e-1$ or $s=-dk+d-e-1$ for some integers $k\ge1$ and $e\in\{1,\ldots,d-1\}$, and we precisely determine the domination ratio of $\Gamma(\ZZ,\{1,2,\ldots,d-2,s\})$.

\begin{theorem}\label{thm:main}
Let $d$ and $s$ be integers with $d\ge2$ and $s\notin[0,d-2]$.
Write $s=dk+e-1$ or $s=-dk+d-e-1$ for some integers $k\ge1$ and $e\in\{1,\ldots,d-1\}$.
Then
\begin{align*}
\dgamma(\ZZ,\{1,2,\ldots,d-2,s\}) 
&=\min \left\{ \frac{k+1}{dk+e},\ \frac{2k+e-1}{2dk-d+2e},\ \frac{1}{d-1} \right\} \\
&=\begin{cases} 
(k+1)/(dk+e) & \text{if } e\ge2,\ d\le k+e+1 \\ 
(2k+e-1)/(2dk-d+2e) & \text{if } e=1,\ d\le 2k+2 \\
1/(d-1) & \text{otherwise}.
\end{cases}
\end{align*}
\end{theorem}

We will give a proof for Theorem~\ref{thm:main} in Section~\ref{sec:main}.
It is analogous to the proof of Theorem~\ref{thm:1s} in earlier work~\cite{DomRatio} but is more complex and requires a more detailed case-by-case analysis.
Two special cases $d=4$ and $d=5$ of Theorem~\ref{thm:main} are explicitly given by Corollary~\ref{cor:d=4} and Corollary~\ref{cor:d=5}.

Theorem~\ref{thm:main} characterizes when an efficient dominating set exists in $\Gamma(\ZZ,\{1,2,\ldots,d-2,s\})$ and gives the domination number of certain circulant graphs (see Corollary~\ref{cor:eff} and Corollary~\ref{cor:finite}).
It also implies 
\[ \dgamma(\ZZ,\{1,2,\ldots,d-2,dk+e-1\}) = \dgamma(\ZZ,\{1,2,\ldots,d-2,-dk+d-e-1\})\]
for any integers $k\ge1$ and $e\in\{1,\ldots,d-1\}$, although we do not have any intuitive explanation for this equality even in the special case of $d=3$.

The paper is structured as follows.
In Section~\ref{sec:prelim} we establish Theorem~\ref{thm:RatioEff} by using periodic sets and reducing of an integer distance graph to a circulant graph.
In Section~\ref{sec:lemmas} we give some lemmas on the density of a set of integers and blocks determined by a dominating set of an integer distance graph.
In Section~\ref{sec:main} we prove Theorem~\ref{thm:main} and draw some corollaries from it.
In Section~\ref{sec:conclusion} we give closing remarks and propose some questions for future research.

\section{Periodic sets and efficient domination}\label{sec:prelim}

In this section we characterize when an integer distance graph $\Gamma(\ZZ,S)$ has an efficient dominating set by looking at its domination ratio.
To achieve this goal, we need to use a `minimum' periodic dominating set and reduce $\ZZ$ and $S$ modulo the period of this set.

A set $U\subseteq \ZZ$ is \emph{periodic} if there exists a positive integer $p$ such that 
\[ U\cap[ip+1,ip+p] = \{ip+j: j\in U\cap[1,p]\}, \quad \forall i\in\ZZ.\]
The smallest such integer $p$ (not necessarily prime) is called the \emph{period} of $U$.
It is easy to calculate the density of a periodic set.

\begin{lemma}\cite{DomRatio}\label{lem:period}
If $U$ is a periodic subset of $\ZZ$ with period $p$ then $\delta(U) = |U\cap[1,p]|/p$.
\end{lemma}

The following result from our previous work~\cite{DomRatio} shows that the domination ratio of $\Gamma(\ZZ,S)$ is achieved by some periodic dominating set $D$, which is an extension of an analogous result on the independence ratio~\cite[Theorem~4]{IndepRatio}.

\begin{proposition}\cite{DomRatio}\label{prop:periodic}
Assume $S$ is a finite subset of $\ZZ\setminus\{0\}$.
Let $a$ and $b$ be the largest nonnegative integers in $S\cup\{0\}$ and $-S\cup\{0\}$, respectively.
Let $c:=a+b$.
Then the domination ratio of $\Gamma(\ZZ,S)$ is achieved by some periodic dominating set $D$ with period $p\le c2^c$.
\end{proposition}

This proposition has a consequence on the domination number of certain circulant digraphs.
 
\begin{proposition}\cite{DomRatio}\label{prop:finite}
Assume $S$ is a finite subset of $\ZZ\setminus\{0\}$.
Let $D$ be a dominating set of $\Gamma(\ZZ,S)$ with period $p$ such that $\dgamma(\ZZ,S) = \delta(D)= |D\cap[1,p]|/p$.
Then $D\cap[1,p]$ is a minimum dominating set of $\Gamma(\ZZ_p,S_p)$ and $\gamma(\ZZ_p,S_p) = |D\cap[1,p]| = \dgamma(\ZZ,S) p$, where $\ZZ_p:=\{1,2,\ldots,p\}$ is the cyclic group of order $p$ under addition modulo $p$ and $S_p\subseteq\ZZ_p$ consists of the least positive residues of all elements in $S$ modulo $p$. 
\end{proposition}


Now we are ready to prove Theorem~\ref{thm:RatioEff}, which allows us to know whether an integer distance graph $\Gamma(\ZZ,S)$ has an efficient dominating set by looking at its domination ratio.

\begingroup
\def\thetheorem{\ref{thm:RatioEff}}
\begin{theorem}
Let $S\subseteq\ZZ\setminus\{0\}$ with $|S|=d<\infty$.
Then the integer distance digraph $\Gamma(\ZZ,S)$ has an efficient dominating set if and only if its domination ratio is $1/(d+1)$.
\end{theorem}
\addtocounter{theorem}{-1}
\endgroup

\begin{proof}
By Proposition~\ref{prop:basic} (iii), if $\Gamma(\ZZ,S)$ has an efficient dominating set, then its domination ratio must be $1/(d+1)$. 
Conversely, assume the domination ratio of $\Gamma(\ZZ,S)$ is $1/(d+1)$.
By Proposition~\ref{prop:periodic}, this ratio is achieved by a periodic dominating set $D$ with period $p$.
By Lemma~\ref{lem:period}, we have $|D\cap[1,p]|/p=1/(d+1)$.
Hence $p=(d+1)k$ where $k:=|D\cap[1,p]|\ge0$ is an integer.

By Proposition~\ref{prop:finite}, the set $D\cap[1,p]$ is a minimum dominating set of the circulant graph $\Gamma(\ZZ_p,S_p)$.
For each $n\in D\cap[1,p]$, the set of all vertices dominated by $n$ in $\Gamma(\ZZ_p,S_p)$ is 
\[ \{n+s: s\in S_p\cup\{p\}\}.\]
Since $|S_p|\le |S|=d$ and $(|S_p|+1)k \le (d+1)k = p$, the set $D\cap[1,p]$ must be an efficient dominating set of $\Gamma(\ZZ_p,S_p)$ and we must have 
\begin{equation}\label{eq:Sp}
|S_p\cup\{p\}|=|S\cup\{0\}|=d+1.
\end{equation}

To show that $D$ is an efficient dominating set of $\Gamma(\ZZ,S)$, suppose that an integer $m$ is dominated by two integers $n, n'\in D$ in $\Gamma(\ZZ,S)$, i.e., $m-n$ and $m-n'$ both belong to $S\cup\{0\}$.
It suffices to show $n=n'$.
Let $r,s,s'\in[1,p]$ be the least positive residues of $m, n, n'$ modulo $p$, respectively.
Then in the group $\ZZ_p$ we have both $r-s$ and $r-s'$ in $S_p\cup\{p\}$.
Since $D\cap[1,p]$ is an efficient dominating set of $\Gamma(\ZZ_p,S_p)$, it contains only one vertex which can dominate $r$. 
This implies $s=s'$ and thus $n\equiv n' \pmod p$.
If $n\ne n'$ then reducing the distinct elements $m-n$ and $m-n'$ of $S\cup\{0\}$ modulo $p$ gives a contradiction to the equation~\eqref{eq:Sp}.
Thus $n=n'$.
\end{proof}

\section{Lemmas on density and blocks}\label{sec:lemmas}

In this section we give some lemmas on density and blocks, which will used in our proof of Theorem~\ref{thm:main}.

Let $f:\ZZ\to\RR$ be a function.
For each nonempty finite set $A\subseteq\ZZ$ we write 
\[ \|f(A)\| := \sum_{a\in A} f(a).\] 
Define the \emph{density} of $f$ to be
\[ \delta(f):=\liminf_{n\to\infty} \frac{ \|f(\ZZ\cap[-n,n])\| }{2n+1}.\]
In particular, for any subset $U\subseteq \ZZ$, let $f=\chi_{\ds U}:\ZZ\to\RR$ be the function defined by 
\[ \chi_{\ds U}(i):=\begin{cases}
1, & i\in U,\\
0, & i\in\ZZ\setminus U.
\end{cases}\]
We define the \emph{density of $U$ in $\ZZ$} to be
\[ \delta(U) := \delta(\chi_{\ds U}) = \liminf_{n\to\infty}  \frac{|U\cap [-n,n]|}{2n+1}. \]
This agrees with the earlier definition~\eqref{eq:density} for the density of a subset of $\ZZ$.

Extending a previous result on the independence ratio~\cite[Lemma 17]{IndepRatio}, the next lemma shows that the density of a subset of $\ZZ$ can be calculated not only by looking over the intervals of the form $[-n,n]$, but also by looking at multiples of these intervals and by making small bounded changes at both ends.

\begin{lemma}\cite{DomRatio}\label{lem:density}
Fix two positive integers $d$ and $N$.
Suppose that $f:\ZZ\to\RR$ satisfies $f(i)\ge0$ for all $i\in \ZZ$.
Let $(\ell_m)$ and $(r_m)$ be two sequences of integers contained in the interval $[-N,N]$.
Then
\[ \delta(f) = \liminf_{m\to\infty} \frac{\| f(\ZZ\cap[-md-\ell_m,md+r_m]) \|}{2md+\ell_m+r_m+1}. \]
\end{lemma}

We will need the following lemma in the proof of Theorem~\ref{thm:main}.

\begin{lemma}\cite{DomRatio}\label{lem:toggle}
If $U\subseteq \ZZ$ and $f:\ZZ\to\RR$ satisfy all of the following conditions, then $\delta(U) = \delta(f)$.
\begin{enumerate}
\item
The set $\ZZ$ is the disjoint union of finite nonempty subsets $A_i$ for $i$ in some index set $I$.
\item\label{item:AiW}
There exists a constant $b$ such that $\max A_i-\min A_i\le b$ for all $i\in I$.
\item\label{item:AiU}
For each $i\in I$ we have $|U\cap A_i| = \|f(A_i)\|$. 
\item\label{item:f}
There exists a constant $N>0$ such that $0\le f(j)\le N$ for all $j\in\ZZ$.
\end{enumerate}
\end{lemma}

Let $S$ be a finite nonempty subset of $\ZZ\setminus\{0\}$.
One can write a dominating set of the integer distance graph $\Gamma(\ZZ,S)$ as $D=\{ x_i: i\in\ZZ \}$, where $x_i< x_{i+1}$ for all $i\in\ZZ$. 
This gives rise to a partition of $\ZZ$ into a disjoint union of \emph{blocks} $B_i:=\{ x_i, x_i+1,\ldots, x_{i+1}-1\}$ for all $i\in \ZZ$.
The size of a block $B_i$ is $b_i:=|B_i| = x_{i+1}-x_i$ and we call $B_i$ a \emph{$b_i$-block}.
The \emph{block structure} of a union of consecutive blocks is defined as the sequence of sizes of the blocks in this union. 
We may identify the dominating set $D$ with the block structure of the above partition of $\ZZ$, since $D$ is determined by this block structure up to a translation. 
If $D$ has period $p$ then we can represent $D$ by a finite sequence $(b_1,\ldots,b_\ell)$ repeated infinitely in both directions, where $b_1,\ldots,b_\ell$ are positive integers satisfying $b_1+\ldots+b_\ell=p$.
For instance, the block structure $(3^5, 4, 3^5, 1)$ corresponds to $12$ consecutive blocks, first with five $3$-blocks, then a $4$-block, then five $3$-blocks again, and finally a $1$-block.
Repeating this block structure infinitely in both directions gives the block structure $(3^5,4,3^5,1)^\infty$, which determines a periodic dominating set up to a translation.

Given integers $d\ge2$ and $s\notin[0,d-2]$, we can use block structures to construct dominating sets for the graph $\Gamma(\ZZ,\{1,2,\ldots,d-2,s\})$ and establish upper bounds for its domination ratio.
One sees that if $s\equiv-1 \pmod d$ then $\Gamma(\ZZ,\{1,2,\ldots,d-2,s\})$ has an efficient domination set with block structure $d^\infty$ (an example of such a set consists of all multiples of $d$) and hence has domination ratio $1/d$ (cf. Proposition~\ref{prop:basic} (iv)). 
The next lemma deals with the remaining cases.

\begin{lemma}\label{lem:1s}
Let $d\ge2$ and $s\notin[0,d-2]$ be integers.
Write $s=dk+e-1$ or $s=-dk+d-e-1$ for some integers $k\ge1$ and $e\in\{1,\ldots,d-1\}$.
Then 
\begin{align}
\dgamma(\ZZ,\{1,2,\ldots,d-2,s\}) \label{eq:upper}
&\le \min \left\{ \frac{k+1}{dk+e},\ \frac{2k+e-1}{2dk-d+2e},\ \frac{1}{d-1} \right\} \\
&=\begin{cases} 
(k+1)/(dk+e) & \text{if } e\ge2,\ d\le k+e+1 \\ 
(2k+e-1)/(2dk-d+2e) & \text{if } e=1,\ d\le 2k+2 \\
1/(d-1) & \text{otherwise}.
\end{cases}\label{eq:cases}
\end{align}
\end{lemma}

\begin{proof}
One sees that $\Gamma(\ZZ,\{1,2,\ldots,d-2,s\})$ has three periodic dominating sets determined by the block structures $(d^k,e)^\infty$, $(d^{k-1}, d+e, d^{k-1}, 1^{e})^\infty$, and $(d-1)^\infty$, respectively.
This gives the desired upper bound~\eqref{eq:upper} for the domination ratio $\dgamma(\ZZ,\{1,2,\ldots,d-2,s\})$.
One can check that 
\begin{itemize}
\item
$(k+1)/(dk+e) \le (2k+e-1)/(2dk-d+2e)$ if and only if $d-e+(dk+e)(e - 2)\ge0$,
\item
$(k+1)/(dk+e) \le 1/(d-1)$ if and only if $d\le k+e+1$, and
\item
$(2k+e-1)/(2dk-d+2e) \le 1/(d-1)$ if and only if $(d-3)e+1\le 2k$.
\end{itemize}
If $e\ge 2$ then $d-e+(dk+e)(e - 2)\ge0$.
If $e=1$ then $d-e+(dk+e)(e - 2) = d(1-k)-2<0$.
Thus the equality~\eqref{eq:cases} holds.
\end{proof}

To establish the equality in the upper bound~\eqref{eq:upper}, we need another lemma on blocks.

\begin{lemma}\label{lem:block}
Let $D=\{x_i:i\in\ZZ\}$ be a dominating set of $\Gamma(\ZZ,\{1,2,\ldots,d-2,s\})$, where $d\ge2$ and  and $s\notin[0,d-2]$ are integers.
Then the following statements hold for every $i\in \ZZ$.
\begin{enumerate}
\item\label{item:size}
The size of $B_i$ satisfies $1\le b_i\le s+1$ if $s>0$ or $1\le b_i\le -s+d-1$ if $s<0$.
\item\label{item:singleton}
If $b_i\ge d$ then $D$ contains $x_i-s+d-1, x_i-s+d,\ldots,x_i-s+b_i-1$.
\end{enumerate}
\end{lemma}

\begin{proof}
The result is trivial when $b_i\le d-1$.
Assume $b_i\ge d$ below.
Then $D$ contains the elements $x_i-s+d-1, x_i-s+d,\ldots,x_i-s+b_i-1$ in order to dominate $x_i+d-1,x_i+d,\ldots,x_i+b_i-1$.
This implies $x_i-s+b_i-1\le x_i$ when $s>0$ or $x_i+b_i \le x_i-s+d-1$ when $s<0$.
The result follows.
\end{proof}

\section{Proof of Theorem~\ref{thm:main}}\label{sec:main}

In this section we precisely determine the domination ratio of the integer distance graph $\Gamma(\ZZ,\{1,2,\ldots,d-2,s\})$ for any integers $d\ge2$ and $s\notin[0,d-2]$, and give some corollaries of this result.
By Proposition~\ref{prop:basic} (iv), we may assume $s\not\equiv-1\pmod d$.

\begingroup
\def\thetheorem{\ref{thm:main}}
\begin{theorem}
Let $d$ and $s$ be integers with $d\ge2$ and $s\notin[0,d-2]$.
Write $s=dk+e-1$ or $s=-dk+d-e-1$ for some integers $k\ge1$ and $e\in\{1,\ldots,d-1\}$.
Then
\begin{align*}
\dgamma(\ZZ,\{1,2,\ldots,d-2,s\}) 
&=\min \left\{ \frac{k+1}{dk+e},\ \frac{2k+e-1}{2dk-d+2e},\ \frac{1}{d-1} \right\} \\
&=\begin{cases} 
(k+1)/(dk+e) & \text{if } e\ge2,\ d\le k+e+1 \\ 
(2k+e-1)/(2dk-d+2e) & \text{if } e=1,\ d\le 2k+2 \\
1/(d-1) & \text{otherwise}.
\end{cases}
\end{align*}
\end{theorem}
\addtocounter{theorem}{-1}
\endgroup

\begin{proof}
We may assume $d\ge4$ as the case $2\le d\le 3$ was solved in our previous work~\cite{DomRatio}.
By Lemma~\ref{lem:1s}, it suffices to show that the desired domination ratio is a lower bound. 
Let $D=\{x_i:i\in\ZZ\}$ be an arbitrary dominating set of $\Gamma(\ZZ,\{1,2,\ldots,d-2,s\})$, where $x_i<x_{i+1}$ for all $i\in \ZZ$.
The set $D$ partitions $\ZZ$ into a disjoint union of blocks $B_i:=\{x_i,x_i+1,\ldots,x_{i+1}-1\}$ for all $i\in \ZZ$.
By Lemma~\ref{lem:block}~\eqref{item:size}, we have $b_i:=|B_i| \le dk+e$.

We next define a function $f:\ZZ\to\RR$ and partition $\ZZ$ into a disjoint union of finite subsets $A_i\ne\emptyset$ for all $i$ in some index set $I$ in such a way that Lemma~\ref{lem:toggle} applies and gives $\delta(D)=\delta(f)$.

\vskip3pt\noindent\textsf{Step 1}. 
We first define $f(x) := 0$ for all $x\in \ZZ\setminus D$ and initiate $I:=\emptyset$.

\vskip3pt\noindent\textsf{Step 2}. 
For each $i\in\ZZ$ with $b_i\le d-1$ we define $f(x_i):=1$, insert $i$ into $I$, and set $A_i:=B_i$.
Then
\[ \max A_i - \min A_i = b_i-1 \le d-2 \quad\text{and}\quad \|f(A_i)\| = 1 = |A_i\cap D|.\]
Since the blocks are pairwise disjoint, at the end of this step we have a disjoint union of $A_i$ for all $i\in I$, which equals the union of all blocks of size less than $d$.

\vskip3pt\noindent\textsf{Step 3}. 
For each $i\in \ZZ$ with $d\le b_i\le dk+e$ we insert $i$ into $I$ and define
\[ w:=\frac{d-2}{2d-2}\in \left(\frac13, \frac12\right) \quad\text{and}\quad
f(x_i) := w(b_i-d)+1 \le \frac{dk+1}2. \]
By Lemma~\ref{lem:block}~\eqref{item:singleton}, there are $b_i-d$ consecutive 1-blocks 
\[ B_j=\{ x_i-s+d-1 \},\ B_{j+1}=\{x_i-s+d\},\ \ldots,\ B_{j+b_i-d-1}=\{x_i-s+b_i-2\}.\]
Let $A_i$ be the union of these 1-blocks together with $B_i$.
Lemma~\ref{lem:block}~\eqref{item:size} implies that
\[ \max A_i - \min A_i = \begin{cases}
(x_i+b_i-1) - (x_i-s+d-1) \le 2dk-d+2e-1 & \text{if } s=dk+e-1, \\
(x_i-s+b_i-2) - x_i \le  2dk-d+2e-1 & \text{if } s=-dk+d-e-1.
\end{cases}\]
Delete $j+h$ from $I$ and redefine $f(x_{j+h}):=1-w$ for all $h=0,1,\ldots,b_i-d-1$. We have
\[ \|f(A_i)\| = (1-w)(b_i-d) + w(b_i-d)+1 = b_i-d+1 = |A_i\cap D|.\]
Since we delete all $1$-blocks included in any set $A_i$ defined in this step, at the end we still have a disjoint union of $A_i$ for all $i\in I$, and this union equals $\ZZ$ as we include all the blocks.

Now for every $i\in\ZZ$ the nonempty set $A_i$ satisfies
\[ \max A_i - \min A_i \le 2dk-d+2e-1 \quad\text{and}\quad \|f(A_i)\| = |A_i\cap D|. \]
For each $x\in\ZZ$ we have 
$ 0\le f(x)\le (dk+1)/2.$
Thus Lemma~\ref{lem:toggle} gives $\delta(D) = \delta(f)$.

It remains to show that the desired domination ratio is a lower bound for $\delta(f)$.
For each integer $n>0$, let $x_r$ be the largest element of $D$ such that $B_r\subseteq(-\infty,n]$ if $s=dk+e-1$, or the smallest element of $D$ such that $B_r\subseteq[-n,\infty)$ if $s=-dk+d-e-1$.
We distinguish some cases below to define a cluster $C_r$ to be a union of certain blocks (including $B_r$) such that
\begin{equation}\label{eq:ClusterDensity}
\frac{\|f(C_r)\|}{|C_r|} \ge \min \left\{ \frac{k+1}{dk+e},\ \frac{2k+e-1}{2dk-d+2e},\ \frac{1}{d-1} \right\}.
\end{equation}

\vskip3pt\noindent\textsf{Case 1}.
If $b_r\le d-1$ then we define a cluster $C_r:=B_r$.
We have $1/2< 1-w\le \|f(C_r)\|\le 1$ when $|C_r|=1$ and $\|f(C_r)\|=1$ when $2\le |C_r|\le d-1$.
Thus $\|f(C_r)\| / |C_r| \ge 1/(d-1)$.

\vskip3pt\noindent\textsf{Case 2}.
If $d\le b_r\le dk+e$ then $D$ contains $x_r-s+d-1, x_r-s+d, \ldots,x_r-s+b_r-1$ 
by Lemma~\ref{lem:block}~\eqref{item:singleton}.
Define
\[ C'_r:=\begin{cases}
[x_r-s+d-1,x_r+b_r-1], & \textrm{if } s=dk+e-1, \\
[x_r,x_r-s+b_r-2], & \textrm{if } s=-dk+d-e-1.
\end{cases} \]
Then $C'_r$ is the disjoint union of all blocks contained in it, and its size is $|C'_r| = b_r+dk-d+e$.
Let $m_\ell$ be the number of blocks of size $\ell$ in $C'_r$ for $1\le\ell\le dk+e$.
We have $m_1\ge b_r-d$ and
\begin{equation}\label{eq:ClusterSize}
m_1+2m_2+\cdots+(dk+e)m_{dk+e} = b_r+dk-d+e.
\end{equation}

\vskip3pt\noindent\textsf{Case 2.1}.
Assume $m_1+\cdots+m_{d-1}\ge1$.
Define a cluster $C_r:=C'_r$ whose size is 
\[ |C_r| = |C'_r| = b_r+dk-d+e \le 2dk-d+2e .\]
To show that $\|f(C_r)\|/|C_r|$ satisfies the lower bound~\eqref{eq:ClusterDensity}, we further distinguish two subcases.

\vskip3pt\noindent\textsf{Case 2.1.1}.
Assume $m_1+m_{d+1}+2m_{d+2}+\cdots+(dk-d+e)m_{dk+e} \le b_r-d+e-1$. 
This implies 
\[ 2m_2+3m_3+\cdots+dm_d+dm_{d+1}+\cdots+dm_{dk+e} \ge dk+1 \]
by the equation~\eqref{eq:ClusterSize}.
Thus $m_2+\cdots+m_{dk+e}\ge k+1$ and
\begin{align*}
\|f(C_r)\| \ge (1-w)m_1 + m_2 + m_3 + \cdots + m_{dk+e} \ge (b_r-d)/2 + k+1.
\end{align*}
It follows that
\[ \frac{\|f(C_r)\|}{|C_r|} \ge \frac{(b_r-d)/2 + k+1}{b_r+dk-d+e} \ge \frac{k+1}{dk+e} \]
where the last inequality holds by the following calculation: 
\begin{align*}
& ((b_r-d)/2 + k+1)(dk+e)-(b_r+dk-d+e)(k+1) \\ 
= & (b_r-d)((d-2)k+e-2)/2 \ge0.
\end{align*}
Hence the inequality~\eqref{eq:ClusterDensity} holds in this case.

\vskip3pt\noindent\textsf{Case 2.1.2}.
Assume $m_1+m_{d+1}+2m_{d+2}+\cdots+(dk-d+e)m_{dk+e} \ge b_r-d+e$.
This implies that
\[ 2m_2+\cdots+dm_d+d m_{d+1}+\cdots+d m_{dk+e} \le dk \]
by the equation~\eqref{eq:ClusterSize}. We further distinguish two subcases.

\vskip3pt\noindent\textsf{Case 2.1.2.1}.
Assume $m_1+\cdots+m_{d-2}\ge1$ or $m_{d-1}\ge2$. 
Since $2/d \ge 1/(d-1)=1-2w$, we have
\begin{align*}
\|f(C_r)\| &\ge (1-w)m_1 + m_2 + m_3 + \cdots + m_{d-1} + \sum_{\ell=d}^{dk+e} (w(\ell-d)+1)m_\ell \\
& = \sum_{\ell=1}^{dk+e} w\ell m_\ell + (1-2w)m_1 + \sum_{\ell=2}^{d-1} (1-\ell w) m_\ell + \sum_{\ell=d}^{dk+e} (1-dw)m_\ell \\
& = \sum_{\ell=1}^{dk+e} w\ell m_\ell + (1-2w)m_1 + \sum_{\ell=2}^{d-1} \frac{d-\ell}{d} m_\ell + \frac{1-dw}{d} \left( \sum_{\ell=2}^{d-1} \ell m_\ell + \sum_{\ell=d}^{dk+e} d m_\ell \right) \\
& \ge w(b_r+dk-d+e) + 1/(d-1) + (1-dw)k \\
& =  w(b_r-d+e) + k+1/(d-1) .
\end{align*}
This implies the lower bound~\eqref{eq:ClusterDensity} for $\|f(C_r)\|/|C_r|$ by the following tedious calculation:
\begin{itemize}
\item
If $e\ge2$ then
\begin{align*}
& (w(b_r-d+e) + k+1/(d-1))(dk+e)-(b_r+dk-d+e)(k+1)  \\
= & \frac{(b_r-d)(dk(d-4) + 2k-2) + (b_r+dk-d+e)(d-2)(e-2) )}{(2d-2)} \ge0.
\end{align*}
\item
If $e=1$ and $d\le 2k+2$ then
\begin{align*}
& (w(b_r-d+e) + k+1/(d-1))(2dk-d+2)-2k(b_r+dk-d+e) \\
= & \frac{(b_r-d)(d(d-4)(2k-1)+4(k-1)) + d(2k+2-d)}{2d-2} \ge0.
\end{align*}
\item
If $e=1$ and $d\ge 2k+2$ then
\begin{align*}
& (w(b_r-d+e) + k+1/(d-1))(d-1)-(b_r+dk-d+e)  \\
= & ((b_r-d)(d-4) + d-2k-2 ) /2 \ge0.
\end{align*}
\end{itemize}

\vskip3pt\noindent\textsf{Case 2.1.2.2}.
Assume $m_1+\cdots+m_{d-2}=0$ and $m_{d-1}\le 1$.
This implies  $b_r=d$ and $m_{d-1}=1$ since $m_1\ge b_r-d$ and $m_1+\cdots+m_{d-1}\ge1$.
Thus $|C_r|=dk+e$ and the equation~\eqref{eq:ClusterSize} becomes
\[ dm_d+(d+1)m_{d+1}+\cdots+(dk+e)m_{dk+e} = dk-d+e+1. \]
Reducing this equation modulo $d$ gives 
\[ m_{d+1}+2m_{d+2}+\cdots+(dk-d+e)m_{dk+e} \ge e+1.\]
since we already have $m_{d+1}+2m_{d+2}+\cdots+(dk-d+e)m_{dk+e} \ge e$.
Similarly as Case 2.1.2.1, we obtain
\begin{align*}
\|f(C_r)\| \ge w(dk+e) + 1/d + (1-dw)(dk-1)/d =  w(e+1)+k.
\end{align*}
This implies the inequality~\eqref{eq:ClusterDensity} since
\[ \frac{\|f(C_r)\|}{|C_r|} \ge 
\begin{cases} 
(k+3w)/(dk+e) \ge (k+1)/(dk+e) & \text{if } e\ge 2 \\
(k+2w)/(dk+e) \ge 2k/(2dk-d+2) & \text{if } e=1
\end{cases} \]
where the last inequality holds by the following calculation:
\[ (k+2w)(2dk-d+2) - 2k(dk+1) =  
\frac{d(d-3)(k-1)+d-4}{d-1} \ge0.\]

\vskip3pt\noindent\textsf{Case 2.2}. 
Assume $m_1+\cdots+m_{d-1}=0$.
Then $b_r=d$ since $0=m_1\ge b_r-d$.
Define a cluster 
\[ C_r := \begin{cases}
[x_r-2s+2d-2,x_r+d-1], & \text{if } s=dk+e-1,\\
[x_r,x_r-2s+d-3], & \text{if } s=-dk+d-e-1.
\end{cases}\]
Then $C_r$ has size $2dk-d+2e$ and is the disjoint union of all blocks in it by the following argument.
\begin{itemize}
\item
If $s=dk+e-1$ then the block containing $x_r-s+d-1$ (the smallest element of $C'_r$) has size at least $d$ since $m_1+\cdots+m_{d-1}=0$, and thus $D$ contains $x_r-2s+2d-2$ to dominate $x_r-s+2d-2$.
\item
If $s=-dk+d-e-1$ then the block containing $x_r-s+d-2$ (the largest element of $C'_r$) has size at least $d$ since $m_1+\cdots+m_{d-1}=0$, and thus $D$ contains $x_r-2s+d-2$ to dominate $x_r-s+d-2$.
\end{itemize}
Let $t_\ell$ be the number of blocks of size $\ell$ contained in $C_r$ for $1\le \ell \le dk+e$.
Then
\[ t_1 + 2t_2 + 3t_3+\cdots+(dk+e)t_{dk+e} = 2dk-d+2e. \]
A block of size $\ell\ge d+1$ in $C'_r$ gives $\ell-d$ many $1$-blocks in $C_r$ by  Lemma~\ref{lem:block}~\eqref{item:singleton}.
Thus
\[ t_1 \ge \sum_{\ell=d+1}^{dk+e} (\ell-d)m_\ell \ge e \]
where the last inequality is obtained by reducing the equation~\eqref{eq:ClusterSize} modulo $d$.
It follows that
\[ 2t_2+3t_3+\cdots+dt_d+dt_{d+1}+\cdots + dt_{dk+e} \le 2dk-d.\]
This implies the inequality~\eqref{eq:ClusterDensity} since we can argue similarly to Case 2.1.2.1 and obtain
\begin{align*}
\frac{\|f(C_r)\|}{|C_r|}
\ge \frac{w(2dk-d+2e) + (1-2w)e + (1-dw)(2dk-d)/d}{2dk-d+2e}
= \frac{2k+e-1}{2dk-d+2e}.
\end{align*}

Now we have the cluster $C_r$ which satisfies the inequality~\eqref{eq:ClusterDensity}.
When $s=dk+e-1$ we recursively write $(-\infty,c-1]$ as a disjoint union of clusters, where $c$ is the smallest element of the cluster $C_r$.
When $s=-dk+d-e-1$ we recursively write $[c+1,\infty)$ as a disjoint union of clusters, where $c$ is the largest element of the cluster $C_r$.
Let $Z_n$ be the union of all clusters contained in $[-n,n]$.
We already know the lower bound~\eqref{eq:ClusterDensity} for each cluster.
Therefore $\|f(Z_n)\| / |Z_n|$ satisfies the same lower bound.
Since every cluster is an interval of size at most $2dk-d+2e$, we have $Z_n= \ZZ\cap [-n+\ell_n,n-r_n]$ for some integers $\ell_n,r_n\in[0,2dk-d+2e-1]$. 
By Lemma~\ref{lem:density}, $\delta(f)$ satisfies the same lower bound as well.
\end{proof}

Specializing $d=4$ and $d=5$ in Theorem~\ref{thm:main} immediately gives the next two corollaries.

\begin{corollary}\label{cor:d=4}
If $s=4k$ or $s=-4k+2$ for some integer $k\ge1$ then $\dgamma(\ZZ,\{1,2,s\})= 2k/(8k-2)$.

If $s = 4k+1$ or $s = -4k+1$ for some integer $k\ge1$ then $\dgamma(\ZZ,\{1,2,s\})= (k+1)/(4k+2)$.

If $s = 4k+2$ or $s = -4k$ for some integer $k\ge1$ then $\dgamma(\ZZ,\{1,2,s\})= (k+1)/(4k+3)$.
\end{corollary}

\begin{corollary}\label{cor:d=5}
If $s\in\{ -3, -2, 5, 6 \}$ then $\dgamma(\ZZ,\{1,2,3,s\})=1/4$.

If $s = 5k$ or $s = -5k+3$ for some integer $k\ge2$ then $\dgamma(\ZZ,\{1,2,3,s\})= 2k/(10k-3)$. 

If $s = 5k+1$ or $s = -5k+2$ for some integer $k\ge2$ then $\dgamma(\ZZ,\{1,2,3,s\})=(k+1)/(5k+2)$.

If $s = 5k+2$ or $s = -5k+1$ for some integer $k\ge1$ then $\dgamma(\ZZ,\{1,2,3,s\})=(k+1)/(5k+3)$.

If $s = 5k+3$ or $s = -5k$ for some integer $k\ge1$ then $\dgamma(\ZZ,\{1,2,3,s\})=(k+1)/(5k+4)$.
\end{corollary}

Theorem~\ref{thm:main} tells us exactly when $\Gamma(\ZZ,\{1,2,\ldots,d-2,s\})$ has an efficient dominating set.

\begin{corollary}\label{cor:eff}
Let $d,s$ be integers with $d\ge2$ and $s\notin[0,d-2]$.
Then there exists an efficient dominating set for $\Gamma(\ZZ,\{1,2,\ldots,d-2,s\})$ if and only if $d=2$ or $s\equiv-1\pmod d$.
\end{corollary}

\begin{proof}
If $d=2$ then we have $\Gamma(\ZZ,\{1,2,\ldots,d-2,s\}) = \Gamma(\ZZ,\{s\})$ with $s\in\ZZ\setminus\{0\}$ and this graph has an efficient dominating set with block structure $(1^{s-1},s+1)^\infty$.
If $s\equiv -1\pmod d$ then the graph $\Gamma(\ZZ,\{1,2,\ldots,d-2,s\})$ has an efficient dominating set with block structure $d^\infty$. 
Finally, if $d\ge 3$ and $s\not\equiv -1\pmod d$ then $\Gamma(\ZZ,\{1,2,\ldots,d-2,s\})$ has  domination ratio less than $1/d$ by Theorem~\ref{thm:main} and thus has no efficient dominating set by Theorem~\ref{thm:RatioEff}. 
\end{proof}

Theorem~\ref{thm:main} also implies the domination number of some circulant graphs.

\begin{corollary}\label{cor:finite}
Let $d\ge2$, $k\ge1$, and $e\ge2$ be integers.

If $d\le k+e+1$ then $\gamma(\ZZ_{dk+e},\{-1,1,2,\ldots,d-2\}) = \gamma(\ZZ_{dk+e},\{1,2,\ldots,d-1\})=k+1$.

If $d\le 2k+2$ then $\gamma(\ZZ_{2dk-d+2},\{1,2,\ldots,d-2,dk\}) =2k$.
\end{corollary}

\begin{proof}
This follows from Proposition~\ref{prop:finite}, the proof of Lemma~\ref{lem:1s}, and Theorem~\ref{thm:main}.
\end{proof}

\section{Conclusion}\label{sec:conclusion}

The (infinite) integer distance graph $\Gamma(\ZZ,S)$ is a natural extension of the (finite) circulant graph $\Gamma(\ZZ_n,S)$. 
In this paper we continue the study of the domination ratio of $\Gamma(\ZZ,S)$ initiated in earlier work~\cite{DomRatio}.
Since the domination ratio $\dgamma(\ZZ,\{1,2,\ldots,d-2,s\})$ given by Theorem~\ref{thm:main} is close to $1/d$ when $k$ is large, we may expect that the domination number $\gamma(\ZZ_n,\{1,2,\ldots,d-2,s\})$ is close to $\lceil n/d \rceil$ when $n$ is large and $s$ is close to $n/2$.
It would be interesting to find the precise value of $\gamma(\ZZ_n,\{1,2,\ldots,d-2,s\})$ and compare with the domination ratio $\dgamma(\ZZ,\{1,2,\ldots,d-2,s\})$.
One could also examine the domination ratio of other families of integer distance graphs in the future.

Finally, the domination ratio of the integer distance digraph $\Gamma(\ZZ,S)$ is related to the problems of integer tiling and sets of integers with missing differences, as explained in Section~\ref{sec:intro}.
The number theoretic methods used in work on these two problems might be helpful to further investigate of the domination ratio of $\Gamma(\ZZ,S)$.

\section*{Acknowlegments}

The author thanks Christian Gaetz for helpful conversations on sets of integers with missing differences, and also thanks an anonymous referee for providing valuable comments and suggestions on the manuscript.

\end{document}